\documentclass[12pt]{article}
\usepackage{amssymb}
\usepackage{amsmath}

\usepackage{stmaryrd}
\usepackage{dsfont}
\usepackage{amsthm}

\usepackage{xcolor}
\usepackage{cancel}
\usepackage[affil-it]{authblk}
 
\usepackage{graphicx}

\newtheorem{theorem}{{Theorem}}
\newtheorem{lemma}[theorem]{{Lemma}}

\newtheorem{remark}[theorem]{{Remark}}
\newtheorem{definition}[theorem]{{Definition}}
\newtheorem{proposition}[theorem]{{Proposition}}
\newtheorem{example}[theorem]{{Example}}

\def\d{{\mathrm d}}

\usepackage{cancel}

\def\M{{\mathcal M}}

\title{Quasi-Lie bialgebroids, Dirac structures,\\and deformations of Poisson quasi-Nijenhuis manifolds}
\date{} 
\author{M.\ do Nascimento Luiz${}^{1}$, I.\ Mencattini${}^1$, M.\ Pedroni${}^{2,3}$}

\affil{
{\small $^1$Instituto de Ci\^encias Matem\'aticas e de Computa\c c\~ao,  Universidade de S\~ao Paulo, Brazil}\\
{\small murilo.luiz@usp.br, igorre@icmc.usp.br 
}\\
\medskip
{\small $^2$Dipartimento di Ingegneria Gestionale, dell'Informazione e della Produzione,  Universit\`a di Bergamo, Italy}\\
{\small marco.pedroni@unibg.it 
}\\
\medskip
{\small $^3$INFN, Sezione di Milano-Bicocca, Piazza della Scienza 3, Milano, Italy}
}

\begin{document}

\maketitle
\abstract{We show how to deform a Poisson quasi-Nijenhuis manifold by means of a closed 2-form. Then we interpret this 
procedure in the context of quasi-Lie bialgebroids, as a particular case of the so called twisting of a quasi-Lie bialgebroid. Finally, we frame our result in the setting of Courant algebroids and Dirac structures.}




\baselineskip=0,6cm

\section{Introduction}
A tensor $N$ 
of type $(1,1)$ on a manifold $\M$ is called a 
Nijenhuis operator 
if its 
Nijenhuis torsion 
vanishes --- see (\ref{tndef1}). Such a geometrical object was introduced in \cite{Nijenhuis51} and is still the subject of interesting investigations 
(see \cite{BKM2022} and references therein). Moreover, it is very useful in the theory of integrable systems thanks to the notion of 
Poisson-Nijenhuis (PN) manifold \cite{KM,MagriMorosiRagnisco85}, i.e., a manifold endowed with a Poisson tensor $\pi$ and a Nijenhuis operator $N$ fulfilling suitable compatibility conditions --- see (\ref{N-P-compatible}). This approach to integrability was put in a more general context in 
\cite{Dorfman1987,Dorfman-book}, using suitable pairs of Dirac structures.

Poisson quasi-Nijenhuis (PqN) manifolds \cite{SX} are a generalization of PN manifolds, where $\pi$ and $N$ are still compatible but the Nijenhuis torsion of $N$ is just required to be controlled by a suitable 3-form (so it is not necessarily zero). Some preliminary results about the relevance 
of PqN manifolds in the study of integrable systems have recently been obtained in \cite{FMOP2020,FMP2023}. Much remains to be done as far as the construction of a family of functions in involution is concerned.

As shown in \cite{YKS96}, one can associate a Lie bialgebroid \cite{MX1994} 
to a PN manifold. In a similar way \cite{SX}, a quasi-Lie bialgebroid \cite{Roytenberg2002} can be associated to a PqN manifold. A general setting to describe Lie bialgebroids and quasi-Lie bialgebroids is supplied by the notions of Courant algebroid and Dirac structure 
\cite{Courant1990,LWX} --- for the definition of Courant algebroid, see Appendix A and references therein.
Indeed, Lie bialgebroids 
are in one-to-one correspondence with pairs of transversal Dirac structures in a Courant algebroid. If one of the transversal structure is not Dirac but Lagrangian, one obtains a quasi-Lie bialgebroid (see, e.g., \cite{SX} and references therein).
Moreover, in \cite{ILX} (following \cite{Roytenberg2002}) it is shown how to deform a quasi-Lie bialgebroid $(A,A^*)$ thanks to the so called twisting by a section of $\bigwedge^2 A$.

The aim of this paper is to present a result about the deformation of a PqN manifold into another PqN manifold, by means of a closed 2-form.  This is a first step to generalize the ideas in \cite{Dorfman-book}, with the aim of constructing (under suitable assumptions) an involutive family of functions on a PqN manifold. We will give a direct proof of our result in Subsection \ref{subsec:PqN}, a proof in the setting of twists of quasi-Lie bialgebroids
in Subsection \ref{subsec:qLba}, and a proof in the context of Courant algebroids and Dirac structures in Section \ref{sec:Courant-Dirac}.

\par\medskip\noindent
{\bf Conventions and notations.} 
Hereafter all manifolds will be smooth of class $C^\infty$ and defined over the real numbers $\mathbb R$. In the same vein, all vector bundles considered will be real and smooth of class $C^\infty$ and their sections will be always considered smooth of the same class. As far as the notation is concerned, a vector bundle whose total space, base and canonical projection are $E$, $\mathcal M$ and, respectively, $p$ will be denoted by the triple $(E,p,\mathcal M)$ or, more simply, by $E$, if this will not be cause of confusion. The space of (global) sections of $E$ will be denoted by $\Gamma(E)$.
For the general notions about Lie algebroids used hereafter we refer the reader to the monograph \cite{libro-Mackenzie}.

\par\medskip\noindent
{\bf Acknowledgments.}
We thank Gregorio Falqui for useful discussions. 
We are grateful to the anonymous referee, whose suggestions helped us to substantially improve the content and the presentation of our manuscript.
MDNL thanks the {\em Department of Mathematics and its Applications\/} of the University of Milano-Bicocca and the {\em Department of Management, Information and Production Engineering\/} of the University of Bergamo for their hospitality, and 
Coordena\c c\~ao de Aperfei\c coamento de Pessoal de N\'ivel Superior - CAPES, Brazil, for supporting him with the grants 88887.506201/2020-00, 88887.695676/2022-00 and  88887.911537/2023-00.
MP thanks the ICMC-UPS, {\em Instituto de Ci\^encias Matem\'aticas e de Computa\c c\~ao\/} of the University of S\~ao Paulo, 
for its hospitality 
and Funda\c c\~ao de Amparo \`a Pesquisa do Estado de S\~ao Paulo - FAPESP, Brazil, for supporting his visit to the ICMC-USP with the grant 2022/02454-8. We also gratefully acknowledge the auspices of the GNFM section of INdAM, under which part of this work was carried out, and the financial support of the MMNLP project - CSN4 of INFN, Italy.



\section{Poisson quasi-Nijenhuis manifolds and quasi-Lie bialgebroids}
In the first part of this section we recall the definition of Poisson quasi-Nijenhuis 
manifold and we present a result, 
generalizing that in \cite{FMP2023}, concerning the deformations of PqN manifolds.
The second part is devoted to a recollection of definitions and results on quasi-Lie bialgebroids, and to an alternative proof of our result.

\subsection{Poisson quasi-Nijenhuis manifolds}
\label{subsec:PqN}

First of all, recall (see, e.g., \cite{KM}) that any bivector $\pi$ on a manifold $\M$ induces a bracket 
on the space $\Gamma(T^*\M)$ of 1-forms, given by
\begin{equation}
\label{eq:liealgpi}
[\alpha,\beta]_\pi=L_{\pi^\sharp\alpha}\beta-L_{\pi^\sharp\beta}\alpha-\d\langle\beta,\pi^\sharp\alpha\rangle,
\end{equation}
where $\pi^\sharp:T^*\M\to T\M$ is defined as $\langle \beta,\pi^\sharp\alpha\rangle=\pi(\alpha,\beta)$, and that (\ref{eq:liealgpi}) is a Lie bracket if and only if $\pi$ is Poisson. Suppose that it is so, and extend 
(\ref{eq:liealgpi}) to a degree $-1$, $\mathbb R$-bilinear bracket on $\Gamma(\bigwedge^{\bullet}T^*\M)$, still denoted by $[\cdot,\cdot]_\pi$, such that for all $\eta\in\Gamma(\bigwedge^qT^\ast\mathcal M)$ and $\eta'\in\Gamma(\bigwedge^{q'}T^\ast\mathcal M)$ 
\begin{itemize}
\item[(K1)] $[\eta,\eta']_\pi=-(-1)^{(q-1)(q'-1)}[\eta',\eta]_\pi$; 
\item[(K2)] $[\alpha,f]_\pi=i_{\pi^\sharp\alpha}\,\d f=\langle \d f,\pi^\sharp\alpha\rangle$ for all $f\in C^\infty(\M)$ and for all 1-forms $\alpha$;
\item[(K3)] 
$[\eta,\cdot]_\pi$ 
is a graded derivation of $(\Gamma(\bigwedge^\bullet T^\ast\mathcal M),\wedge)$, that is, 
for any differential form $\eta''$,
\begin{equation}
\label{deriv-koszul}
[\eta,\eta'\wedge\eta'']_\pi=[\eta,\eta']_\pi\wedge\eta''+(-1)^{(q-1)q'}\eta'\wedge[\eta,\eta'']_\pi.
\end{equation}
\end{itemize}
It follows that the graded Jacobi identity,
\begin{equation}
\label{graded-jacobi}
(-1)^{(q_1-1)(q_3-1)}[\eta_1,[\eta_2,\eta_3]_\pi]_\pi+(-1)^{(q_2-1)(q_1-1)}[\eta_2,[\eta_3,\eta_1]_\pi]_\pi+(-1)^{(q_3-1)(q_2-1)}[\eta_3,[\eta_1,\eta_2]_\pi]_\pi=0,
\end{equation}
holds (see, e.g., \cite{Marle2008}, Proposition 5.4.9), where $q_i$ is the degree of $\eta_i$, 
and that, for any differential form $\eta$ and for any $f\in C^\infty(\M)$, 
\begin{equation}
\label{K2pertutte}
[f,\eta]_\pi=i_{\pi^\sharp \d f}\,\eta.
\end{equation}
\begin{remark}\label{rem:remark1}
The bracket on $\Gamma(\bigwedge^\bullet T^\ast\mathcal M)$ introduced above is an instance of the so called \emph{generalized Schouten bracket}, which can be defined on the exterior algebra of the vector space $\Gamma(A)$ of any Lie algebroid $(A,[\cdot,\cdot]_A,\rho_A)$, see pages 418-419 in \cite{MX1994}. The generalized Schouten bracket associated to the Lie algebroid $(T\mathcal M,[\cdot,\cdot],\text{id})$ is the usual Schouten bracket defined on $\Gamma(\bigwedge^\bullet T\mathcal M)$, while the one stemming from (\ref{eq:liealgpi}) is associated to the Lie algebroid defined by the Poisson bivector field $\pi$.
Note that the sign conventions in the above formulas are the same as in \cite{MX1994} and in agreement with the ones adopted in 
\cite{Dorfman-book}, see Proposition 2.13. Indeed, if read at the level of the generalized Schouten bracket $[\cdot,\cdot]_\pi$, it becomes, 
for 1-forms $\alpha_i$ and $\beta_j$, 
\begin{eqnarray}
&&[\alpha_1\wedge\cdots\wedge\alpha_n,\beta_1\wedge\cdots\wedge\beta_m]_\pi\nonumber\\
&&\qquad=\sum_{i=1}^n\sum_{j=1}^m(-1)^{i+j}[\alpha_i,\beta_j]_\pi\wedge\alpha_1\wedge\cdots\wedge\hat{\alpha}_i\wedge\cdots\wedge\alpha_n\wedge\beta_1\wedge\cdots\wedge\hat{\beta}_j\wedge\cdots\wedge\beta_m,\nonumber
\end{eqnarray}
that can be obtained applying iteratively (K1) and (K3) to its left-hand side.
\end{remark}

The following example points to a useful identity.
\begin{example}
If $\Omega\in\Gamma(\bigwedge^2T^\ast\mathcal M)$ and $\Omega^\flat:T\M\to T^*\M$ is defined as usual by 
$\Omega^\flat(X)=i_X\Omega$, then for all $X,Y,Z\in\Gamma (T\mathcal M)$
\begin{equation}
\label{eq:dorfman}
[\Omega,\Omega]_\pi(X,Y,Z)=2\sum_{\circlearrowleft(X,Y,Z)}\left(\langle[\Omega^\flat X,\Omega^\flat Y]_\pi,Z\rangle)
-{\mathcal L}_{\pi^\sharp\Omega^\flat X}(\Omega(Y,Z))\right),
\end{equation}
which follows by applying to our setting Formula (2.31) and Proposition 2.15 in \cite{Dorfman-book}.
\end{example}
%
%
%
%
%

Recall now that a Poisson tensor $\pi$ and a $(1,1)$ tensor field $N:T\M\to T\M$ on $\M$ are said to be {\it compatible\/} \cite{KM,MagriMorosiRagnisco85} if
\begin{equation}
\label{N-P-compatible}
N\pi^\sharp=\pi^\sharp N^*
\quad\mbox{and}\quad
[\alpha,\beta]_{\pi_N}=[N^*\alpha,\beta]_{\pi}+[\alpha,N^*\beta]_{\pi}-N^*[\alpha,\beta]_{\pi}\ 
\mbox{for all 1-forms $\alpha,\beta$,} 
\end{equation}
where $N^*:T^*\M\to T^*\M$ is the transpose of $N$ and 
$\pi_N$ is the bivector field defined by $\pi_N^\sharp=N\pi^\sharp$ (notice that it is a bivector thanks to the first condition).

Finally, given a $p$-form $\alpha$ with $p\ge 1$, the $p$-form $i_N\alpha$ is defined as 
\begin{equation}
\label{iNalpha}
i_N\alpha(X_1,\dots,X_p)=\sum_{i=1}^p \alpha(X_1,\dots,NX_i,\dots,X_p),
\end{equation}
while $i_N f=0$ for all functions $f$. After introducing 
\begin{equation}
\label{eq:dNd}
\d_N=i_N\circ \d-\d\circ i_N,
\end{equation}
it can be proved \cite{YKS96} that the compatibility between a Poisson tensor $\pi$ and a $(1,1)$ tensor field $N$ is equivalent 
to requiring that $\d_N$ is a derivation of $[\cdot,\cdot]_\pi$, a fact that we will often use in the rest of the paper. Moreover, for any $k$-form 
$\eta$ one has that 
\begin{equation}
\label{def:d_N}
\begin{aligned}
	(\d_N\eta)(X_1,\ldots,X_{k+1})&= 
	\sum_{i=1}^{k+1}(-1)^{i+1}{\mathcal L}_{NX_i} \left( {\eta(X_1,\ldots,\hat{X_i},\ldots,X_{k+1})} \right) \\
	&+\sum_{i <j} (-1)^{i+j}\eta([X_i,X_j]_N, X_1 \ldots, \hat{X_i}, \ldots ,\hat{X_j}, \ldots , X_{k+1}),
\end{aligned}
\end{equation}
where
\begin{equation}
\label{bracketNdef}
[X,Y]_N=[NX,Y]+[X,NY]-N[X,Y].
\end{equation}

In \cite{SX} a {\it Poisson quasi-Nijenhuis (PqN) manifold\/} was defined as a quadruple $(\M,\pi,N,\phi)$ such that:
\begin{itemize}
\item the Poisson bivector $\pi$ and the $(1,1)$ tensor field $N$  
are compatible;
\item the 3-forms $\phi$ and $i_N\phi$ are closed;
\item $T_N(X,Y)=\pi^\sharp\left(i_{X\wedge Y}\phi\right)$ for all vector fields $X$ and $Y$, where $i_{X\wedge Y}\phi$ is the 1-form defined as $\langle i_{X\wedge Y}\phi,Z\rangle=\phi(X,Y,Z)$, and  
\begin{equation}
\label{tndef1}
T_N(X,Y)=[NX,NY]-N[X,Y]_N
\end{equation}
is the {\it Nijenhuis torsion\/} of $N$.
\end{itemize}
A slightly more general definition of PqN manifold was recently proposed in \cite{BursztynDrummond2019} --- 
see also \cite{BursztynDrummondNetto2021}, where the PqN structures are recast in the more general framework of the Dirac-Nijenhuis ones.
Another interesting generalization, given by the so called PqN manifolds with background, was considered in
\cite{Antunes2008,C-NdC-2010}.

If $\phi=0$, then the torsion of $N$ vanishes and $\M$ becomes a {\it Poisson-Nijenhuis (PN) manifold} (see \cite{KM} and references therein). 
In this case, 
$\pi_N$ 
is a Poisson tensor compatible with $\pi$, so that $(\M,\pi,\pi_N)$ is a bi-Hamiltonian manifold. 

The following theorem generalizes a result in \cite{FMP2023}, where the starting point was a PN manifold. 
\begin{theorem}
\label{thm:gim}
Let $(\M,\pi,N,\phi)$ be a PqN manifold and let $\Omega$ be a closed 2-form.
If $\widehat N=N+\pi^\sharp\,\Omega^\flat$
and 
\begin{equation}
\label{phi-hat}
\widehat\phi=\phi+\d_N\Omega+\frac{1}{2}[\Omega,\Omega]_\pi,
\end{equation} 
then $(\M,\pi,\widehat N,\widehat\phi)$ is a PqN manifold. In particular, if $\Omega$ is a solution of the 
non homogeneous Maurer-Cartan equation
\begin{equation}\label{eq:homeq}
\d_N\Omega+\frac{1}{2}[\Omega,\Omega]_\pi=-\phi,
\end{equation}
then $(\M,\pi,\widehat N)$ is a PN manifold.
\end{theorem}\noindent
\begin{proof} It is similar to the one given in \cite{FMP2023}, corresponding to the case $\phi=0$. We present here only the main points, focussing on the differences.
  
To prove the compatibility between $\pi$ and $\widehat N$, we notice that $\d\Omega=0$ implies
\begin{equation}
\label{eq:monella}
\d_{\pi^\sharp\,\Omega^{\flat}}=[\Omega,\cdot]_\pi.
\end{equation}
In fact, both $\d_{\pi^\sharp\,\Omega^{\flat}}$ and $[\Omega,\cdot]_\pi$ are graded derivation of $(\Omega^\bullet(\M),\wedge)$, anticommuting with $\d$ and coinciding on $C^\infty(\M)$ --- see  \cite{FMP2023} for details. The previous identity entails that
$$
\d_{\widehat N}=\d_{N+\pi^\sharp\,\Omega^{\flat}}=\d_N+[\Omega,\cdot]_\pi.
$$
Hence $\d_{\widehat N}$ is a derivation of $[\cdot,\cdot]_\pi$, yielding the compatibility between $\widehat N$ and $\pi$.

The closedness of $\widehat\phi$ easily follows from that of $\phi$ and $\Omega$, recalling that $\d\circ \d_N=-\d_N\circ \d$ and that $\d$ is a derivation of $[\cdot,\cdot]_\pi$. Moreover,  
\begin{equation}\label{eq:1d}
\begin{aligned}
\d_{\widehat N}\widehat\phi&=(\d_N+[\Omega,\cdot]_\pi)(\phi+\d_N\Omega+\frac{1}{2}[\Omega,\Omega]_\pi)\\
&=\d_N\phi+\d_N^2\Omega+
\frac12[\d_N\Omega,\Omega]_\pi-\frac12[\Omega,\d_N\Omega]_\pi+[\Omega,\phi]_\pi
+[\Omega,\d_N\Omega]_\pi+\frac{1}{2}[\Omega,[\Omega,\Omega]_\pi]_\pi\\
&=[\phi,\Omega]_\pi+[\d_N\Omega,\Omega]_\pi+[\Omega,\phi]_\pi
+[\Omega,\d_N\Omega]_\pi+\frac{1}{2}[\Omega,[\Omega,\Omega]_\pi]_\pi=0,
\end{aligned}
\end{equation} 
thanks to $\d_N^2=[\phi,\cdot]_\pi$, the fact that $\d_N$ is a derivation of $[\cdot,\cdot]_\pi$, the commutation rule (K1), and the graded Jacobi identity (\ref{graded-jacobi}). Furthermore, observe that $\d_N\phi=0$ as a consequence of (\ref{eq:dNd}).
Then $i_{\widehat N}\widehat\phi$ is closed because of (\ref{eq:dNd}) and $\d\widehat\phi=0$. 

Finally, if $\alpha$ is any differential form, 
\begin{equation}
\label{d_N^2}
\begin{aligned}
\d_{\widehat N}^2\alpha&=(\d_N+[\Omega,\cdot]_\pi)(\d_N\alpha+[\Omega,\alpha]_\pi)\\
\noalign{\medskip}
&=\d_N^2\alpha+[\d_N\Omega,\alpha]_\pi-[\Omega,\d_N\alpha]_\pi+[\Omega,\d_N\alpha]_\pi+[\Omega,[\Omega,\alpha]_\pi]_\pi\\
&=
[\phi,\alpha]_\pi+[\d_N\Omega,\alpha]_\pi+\frac{1}{2}[[\Omega,\Omega]_\pi,\alpha]_\pi\\
&=[\widehat\phi,\alpha]_\pi.
\end{aligned}
\end{equation}
As shown in Lemma 3.7 of \cite{SX}, this implies that
$T_{\widehat N}(X,Y)=\pi^\sharp\left(i_{X\wedge Y}\widehat\phi\right)$ 
for all vector fields $X,Y$.
\end{proof}

\begin{remark}\label{rem:close1} If $\mathfrak G_{\Omega^2_c}$ is the (additive) group of closed 2-forms on $\M$, the theorem above implies that the set of the PqN-structures on $\M$ carries the following $\mathfrak G_{\Omega^2_c}$-action: 
\begin{equation}
\Omega.(\M,\pi,N,\phi)=(\M,\pi,N+\pi^\sharp\Omega^\flat,\phi+\d_N\Omega+\frac{1}{2}[\Omega,\Omega]_\pi).\label{eq:act1}
\end{equation}
In fact, if $\Omega_1,\Omega_2\in\mathfrak G_{\Omega^2_c}$, then 
\begin{eqnarray*}
&&\Omega_1.(\Omega_2.(\M,\pi,N,\phi))\\
&&\qquad=\Omega_1.(\M,\pi,N+\pi^\sharp\Omega_2^\flat,\phi+\d_N\Omega_2+\frac{1}{2}[\Omega_2,\Omega_2]_\pi))\\
&&\qquad=(\M,\pi,N+\pi^\sharp(\Omega_1+\Omega_2)^\flat,\phi+\d_N\Omega_2+\frac{1}{2}[\Omega_2,\Omega_2]_\pi+\d_{N+\pi^\sharp\Omega_2^\flat}\Omega_1+\frac{1}{2}[\Omega_1,\Omega_1]_\pi)\\
&&\qquad=(\M,\pi,N+\pi^\sharp(\Omega_1+\Omega_2)^\flat,\phi+\d_N(\Omega_1+\Omega_2)+\frac{1}{2}[\Omega_1+\Omega_2,\Omega_1+\Omega_2]_\pi),
\end{eqnarray*}
where the last equality was obtained observing that $\d_{N+\pi^\sharp\Omega_2}=\d_N+\d_{\pi^\sharp\Omega_2}=\d_N+[\Omega_2,\cdot]_\pi$ since $\d\Omega_2=0$, and $[\Omega_1,\Omega_2]_\pi=[\Omega_2,\Omega_1]_\pi$, see Property (K1).
\end{remark}

The aim of the rest of the paper is two give two alternative proofs of Theorem \ref{thm:gim}, in the frameworks of quasi-Lie bialgebroids (see next subsection) and Dirac structures (see Section \ref{sec:Courant-Dirac}).

\subsection{Quasi-Lie bialgebroids}
\label{subsec:qLba}
Suppose that $p:A\to\M$ is a Lie algebroid (see, e.g., \cite{libro-Mackenzie}) with anchor $\rho_A:A\to T\M$ and Lie bracket $[\cdot,\cdot]_A$, 
defined on the space of sections $\Gamma(A)$ and then extended to $\Gamma(\bigwedge^\bullet A)$, see Remark \ref{rem:remark1}. 
Recall that, given $P\in\Gamma(\bigwedge^k A^*)$, 
one can define $\d_A P\in\Gamma(\bigwedge^{k+1} A^*)$ as 
\begin{equation}
\label{def:d_A}
\begin{aligned}
	(\d_A P)(\alpha_1,\ldots,\alpha_{k+1})&= 
	\sum_{i=1}^{k+1}(-1)^{i+1}\rho_A(\alpha_i) \left( {P(\alpha_1,\ldots,\hat{\alpha_i},\ldots,\alpha_{k+1})} \right) \\
	&+\sum_{i <j} (-1)^{i+j}P([\alpha_i,\alpha_j]_A, \alpha_1 \ldots, \hat{\alpha_i}, \ldots ,\hat{\alpha_j}, \ldots , \alpha_{k+1}),
\end{aligned}
\end{equation}
for all $\alpha_1,\ldots,\alpha_{k+1} \in \Gamma(A)$, and that $\d_A$ is a degree-1 derivation of $\Gamma(\bigwedge^\bullet A^*)$ such that $\d_A^2=0$.
\begin{definition}[\cite{SX}]
\label{def:qLba}
A quasi-Lie bialgebroid is a triple $(A,\d_{A^*}, \phi)$, where
\begin{itemize} 
\item $A$ is a Lie algebroid
\item $\d_{A^*}$ is a degree-1 derivation of 
$\Gamma(\bigwedge^{\bullet}A)$, both with respect to the wedge product and 
$[\cdot,\cdot]_A$
\item $\phi \in \Gamma(\bigwedge^3 A)$ satisfies $\d_{A^*} \phi =0$ and $\d_{A^*}^2=[\phi, \cdot]_A$.
\end{itemize} 
\end{definition}
Under these assumptions, one can define a morphism $\rho_{A^*}\colon A^* \to T\M$ by 
\begin{equation}\label{eq:Astaranchor}
\rho_{A^*}(X 
)(f)=X  
(\d_{A^*}f),\qquad \forall 
X\in \Gamma(A^{*}), \, 
f \in C^{ \infty}(\M), 
\end{equation}
a bracket in $\Gamma(A^*)$ by 
\begin{equation}\label{eq:Astarbracket}
[X,Y]_{A^*}(\alpha)= \rho_{A^*}(X)(\alpha(Y))-\rho_{A^*}(Y)(\alpha(X)) - (\d_{A^*}\alpha)(X,Y),\quad \forall \alpha \in \Gamma(A),
\end{equation}
and show that $\d_{A^*}$ is explicitly given by a formula which is analog to (\ref{def:d_A}).
If $\phi=0$, then $(A^*,\rho_{A^*},[\cdot,\cdot]_{A^*})$ is also a Lie algebroid, and $(A,A^*)$ turns out to be a Lie bialgebroid (see, e.g., \cite{SX}). 

The first alternative proof of Theorem \ref{thm:gim} hinges on the following two results.
\begin{proposition}[\cite{SX}, Proposition 3.5]\label{pro:PXS}
The quadruple $(M,\pi,N,\phi)$ is a Poisson quasi-Nijenhuis manifold if and only if $((T^\ast M)_\pi,\d_N,\phi)$ is 
a quasi-Lie bialgebroid and $\phi$ is closed.
\end{proposition}
\begin{proposition}\label{pro:liequa}
If $(A,\d_{A^*}, \phi)$ is a quasi-Lie bialgebroid, $\Omega\in \bigwedge^{2}A$, and 
\[
{\widehat\d}_{A^*}={\d}_{A^*}+[\Omega,\cdot]_A,\qquad \widehat\phi=\phi+\d_A\Omega+\frac12[\Omega,\Omega]_A,
\]
then $(A,{\widehat\d}_{A^*},\widehat\phi)$ is a quasi-Lie bialgebroid too, called the \emph{twist} of $(A,\d_{A^*}, \phi)$ by $\Omega$.
\end{proposition}
\noindent This statement can be found, for example, at the beginning of Section 4.4 of \cite{ILX}. Its proof consists in showing that the new defined triple satisfies what is needed to form a quasi-Lie bialgebroid, i.e., ${\widehat\d}_{A^*}\widehat\phi=0$ and ${\widehat\d}_{A^*}^2=[\widehat\phi,\cdot]_{A}$, which can be both checked by a direct computation completely analogous to (\ref{eq:1d}) and, respectively, (\ref{d_N^2}).
\begin{remark}
Note that if $\phi=0$ Proposition \ref{pro:PXS} reduces to the correspondence between Lie bialgebroids and PN manifolds described in \cite{YKS96}.
\end{remark}

Going back to the proof Theorem \ref{thm:gim}, let $((T^*\M)_\pi,\d_N,\phi)$ be the quasi-Lie bialgebroid associated to the PqN manifold $(\M,\pi,N,\phi)$, and let $\Omega$ be
any 2-form on $\M$. Then we obtain the twist $((T^*\M)_\pi,\d_N+[\Omega,\cdot]_\pi,\widehat\phi)$, where
$$
\widehat\phi=\phi+\d_N\Omega+\frac12[\Omega,\Omega]_\pi.
$$
If $\d\Omega=0$, we know from (\ref{eq:monella}) that
$[\Omega,\cdot]_\pi=\d_{\pi^\sharp\,\Omega^{\flat}}$. Therefore
$$
\d_N+[\Omega,\cdot]_\pi=\d_{N+\pi^\sharp\,\Omega^{\flat}}=\d_{\widehat N},\qquad\mbox{where $\widehat N=N+\pi^\sharp\,\Omega^{\flat}$}.
$$
The last step is to realize that the quasi-Lie bialgebroid $\left((T^*\M)_\pi,\d_{\widehat N},\widehat\phi\right)$ comes from a PqN manifold, i.e., 
that $\widehat\phi$ is closed. 
This follows, as already observed just above Formula (\ref{eq:1d}), from the fact that 
both $\phi$ and $\Omega$ are closed.

\begin{remark}
Proposition \ref{pro:liequa} and a computation very similar 
to the one in Remark \ref{rem:close1} show that the (additive) group $\mathfrak G_{\Omega^2}$ of all 2-forms on $\M$ acts on the set of all quasi-Lie bialgebroids of the form $((T^\ast\M)_\pi,\d_N,\phi)$ by \emph{twisting} them, i.e.,
\begin{equation}
\Omega.((T^\ast\M)_\pi,\d_N,\phi)=((T^\ast\M)_\pi,\d_N+[\Omega,\cdot]_\pi,\phi+\d_N\Omega+\frac{1}{2}[\Omega,\Omega]_\pi).\label{eq:act2}
\end{equation}
The latter restricts to an action of $\mathfrak G_{\Omega^2_c}$, see Remark \ref{rem:close1}, on the set of all quasi-Lie bialgebroids of the form 
$((T^\ast\M)_\pi,\d_N,\phi)$ with $\d\phi=0$. Finally, the bijection between the set of PqN structures $(\M,\pi,N,\phi)$ and the quasi-Lie bialgebroids of the form $((T^\ast\M)_\pi,\d_N,\phi)$ invoked in Proposition \ref{pro:PXS} intertwines (\ref{eq:act2}) with (\ref{eq:act1}).
\end{remark}

\section{Courant algebroids and Dirac structures}
\label{sec:Courant-Dirac}

In this section we give a second alternative proof of Theorem \ref{thm:gim} in the framework of Dirac structures. 
To this end, we start by recalling 
how quasi-Lie bialgebroids are related to Courant algebroids, see \cite{Roytenberg2002,SX} --- 
more precisely, how quasi-Lie bialgebroids correspond to the pairs formed by a Dirac structure and a complementary Lagrangian 
subbundle in a given Courant algebroid, see the proof of part (ii) of Theorem 2.6 in \cite{SX}.
\noindent For the reader convenience, we recall the definition of a Courant algebroid in Appendix A.

Given a quasi-Lie bialgebroid $(A,\d_{A^*}, \phi)$, and using the notations in and after Definition \ref{def:qLba}, 
we consider $E= A^* \oplus A$ together with the non degenerate symmetric pairing 
$$
 \langle X_{1}+
 \alpha_1,X_{2}+
 \alpha_{2} \rangle_E=\dfrac{1}{2}\left( X_{1}\left( \alpha_{2}\right) +X_{2}\left( \alpha_{1}\right) \right),
$$
the bundle map $\rho \colon E \to TM$ given by
$$
\rho(X+ 
\alpha)= \rho_{A^*}(X) + \rho_{A}(\alpha),
$$
and the bracket in $\Gamma(E)$ defined by
\begin{equation}
\begin{aligned}
\label{Courant-bracket}
    \llbracket \alpha,\beta \rrbracket&= [\alpha,\beta]_{A}\\
    \noalign{\medskip}
    \llbracket X,Y \rrbracket &=   \left[ X,Y\right] _{A^{\ast }}+ 
    i_{X\wedge Y}\phi\\
    \llbracket \alpha,Y \rrbracket&= 
    \left( i_{\alpha} (\d_{A}Y)  + \frac{1}{2}\d_{A}(Y(\alpha)) \right) - 
    \left( i_{Y} (\d_{A^*}\alpha) + \frac{1}{2}\d_{A^*}(Y(\alpha)) \right),
\end{aligned}
\end{equation}
for all $X,Y\in \Gamma(A^*)$ and $\alpha,\beta\in \Gamma(A)$, where the bracket $[\cdot,\cdot]_{A^\ast}$ was defined in (\ref{eq:Astarbracket}) while $\d_A$ is the differential defined on  $\Gamma(\bigwedge^\bullet A^\ast)$ by the algebroid $A$.  Then 
$(E,\langle\cdot,\cdot\rangle_E, \llbracket \cdot,\cdot\rrbracket,\rho)$ is a Courant algebroid, $A$ is a Dirac structure (i.e., it is maximal isotropic and its space of sections is closed under $\llbracket \cdot,\cdot\rrbracket$), 
and $A^*$ is a Lagrangian (i.e., maximal and isotropic) subbundle of $E$. Note that, if $\phi=0$, i.e., in the case of a Lie bialgebroid, $A^*$ is a Dirac structure too.

On the other hand, let $(E,\langle \cdot , \cdot \rangle_E ,\llbracket \cdot , \cdot  \rrbracket, \rho)$ be a Courant algebroid and let $A\subset E$ be a Dirac structure (it follows that $A$ has an induced Lie algebroid structure). Suppose that there exists a Lagrangian subbundle $L$ which is transversal to $A$. 
Then we can identify $A^*$ with  
$L$ through
    \begin{equation}
    \label{eq:identification}
    \begin{aligned}
        A^* &\to L\\
        X &\mapsto 
        \tilde X
    \end{aligned}
    \end{equation}
where $
X({\alpha}) =  2  \langle\tilde X, \alpha \rangle_E$ for all $\alpha\in A$. Since we have the identification 
$\Gamma(\bigwedge^\bullet A)\simeq \Gamma(\bigwedge^\bullet L^*)$, we can define 
$\d_L:\Gamma(\bigwedge^\bullet A)\to \Gamma(\bigwedge^\bullet A)$ as in (\ref{def:d_A}), that is,
\begin{equation}
\label{eq:dL}
\begin{aligned}
	(\d_L 
	\psi)(X_1,\ldots,X_{k+1})&= \sum_{i=1}^{k+1}(-1)^{i+1}\rho({\tilde X}_i) \left(\tilde
	\psi({\tilde X}_1,\ldots,\hat{{\tilde X}}_i,\ldots,{\tilde X}_{k+1}) \right) \\
	&+\sum_{i <j} (-1)^{i+j}
	\tilde\psi([{\tilde X}_i,{\tilde X}_j]_L, {\tilde X}_1 \ldots, \hat{{\tilde X}}_i, \ldots ,\hat{{\tilde X}}_j, \ldots , {\tilde X}_{k+1}),
\end{aligned}
\end{equation}
for all $X_1,\ldots,X_{k+1} \in 
\Gamma(A^*)$, where 
$\psi\in\Gamma(\bigwedge^k A)$ corresponds to $\tilde\psi\in\Gamma(\bigwedge^k L^*)$, 
and $[\cdot,\cdot]_L$ is the $L$-component of $\llbracket \cdot,\cdot\rrbracket$ with respect to the splitting $E= L \oplus A$.
Notice that $\d_L$ is a derivation, but $\d_L^2\ne 0$ since $L$ in general is not a Dirac structure.
Finally, let $\varphi \in \Gamma(\bigwedge^3 A)$ be defined 
as
\begin{equation}
\label{eq:phi-Courant}
\varphi(X_1,X_2,X_3)=2 \langle \llbracket {\tilde X}_1, {\tilde X}_2 \rrbracket , {\tilde X}_3 \rangle_E
\qquad \text{ for all } X_1,X_2,X_3 \in 
\Gamma(A^*). 
\end{equation}
Then $(A,\d_L,\varphi)$ turns out to be a quasi-Lie bialgebroid. 

\begin{example}
Consider the quasi-Lie bialgebroid $(T^\ast\M,\d,\phi)$, where $A=T^\ast\M$ with its trivial Lie algebroid structure, i.e., with zero bracket and zero anchor, $\d$ is the Cartan differential and $\phi$ is any closed 3-form on $\M$. Applying (\ref{Courant-bracket}) to the this setting, 
one finds
\[
\begin{aligned}
    \llbracket \alpha,\beta \rrbracket&= 0\\
    \noalign{\medskip}
    \llbracket X,Y \rrbracket &=   \left[ X,Y\right] + 
    i_{X\wedge Y}\phi\\
    \llbracket \alpha,Y \rrbracket&= - 
    \left( i_{Y} (\d\alpha) + \frac{1}{2}\d(Y(\alpha)) \right),
\end{aligned}
\]
for all $X,Y\in\Gamma(T\M)$ and $\alpha,\beta\in\Gamma(T^\ast\M)$, since, in that case, $\d_{A}=0$, $[\cdot,\cdot]_A=0$, $\d_{A^\ast}=\d$ 
and $[\cdot,\cdot]_{A^\ast}=[\cdot,\cdot]$ 
is the usual Lie bracket defined on $\Gamma(T\M)$. Computing the bracket so obtained on a pair of general sections of $E$, one has
\begin{eqnarray}\label{eq:twisted}
\llbracket X+\alpha,Y+\beta\rrbracket
=[X,Y]+i_{X} (\d\beta)-i_{Y} (\d\alpha) + \frac{1}{2}(\d(X(\beta)-\d(Y(\alpha))+i_{X\wedge Y}\phi,
\end{eqnarray}
i.e., $E$ is the Courant algebroid $\mathbb T\M^\phi$, obtained twisting the standard one by the closed 3-form $\phi$, see Example \ref{ex:standardCou}. Note that $T\M$ is a Lagrangian subbundle of $E$ transversal to the Dirac structure $T^\ast\M\subset E$. In particular $(T\M,T^\ast\M)$ is the pair corresponding to $(T^\ast\M,\d,\phi)$. Note now that if $\Omega$ is any closed 2-form in $M$, its graph 
\begin{equation}\label{eq:graphomega}
{\rm Gr}(\Omega)=\{X + \Omega^\flat X\in T\M\oplus T^\ast\M\mid X\in\Gamma(T\M)\}
\end{equation}
is a Lagrangian subbundle of $\mathbb T\M^\phi$, transversal to $T^\ast\M$. A simple computation, 
using (\ref{eq:dL}) and (\ref{eq:phi-Courant}), shows that the quasi-Lie bialgebroid corresponding to $({\rm Gr}(\Omega),T^\ast\M)$ is 
$(T^\ast\M,\d,\phi)$. 
\end{example}

\begin{remark}\label{rem:actLag}
It is worth recalling that every closed 2-form $\Omega$ defines an automorphism of $\mathbb T\mathcal M^\phi$ via the 
formula
\begin{equation}
\Omega.(X+\alpha)=X+i_X\Omega+\alpha,\;\;\forall X+\alpha\in\Gamma(\mathbb T\mathcal M^\phi),\label{eq:gauge}
\end{equation}
see for example \cite{Gualtieri}, just above Proposition 2.2, the proof of Lemma 3.1 in \cite{Bressler} and, in a more general setting, 
Remark 4.2 in \cite{Roytenberg2002}. In particular, the image of a Dirac structure under (\ref{eq:gauge}) is still a Dirac structure. More in general, if $\Omega$ is any 2-form on $\mathcal M$, (\ref{eq:gauge}) sends a Lagrangian subbundle to a Lagrangian subbundle and a pair of transversal Lagrangian subbundles to a pair of transversal Lagrangian subbundles. Note that, in the previous example, the pair $({\rm Gr}(\Omega),T^\ast\mathcal M)$ is obtained applying (\ref{eq:gauge}) to the pair of transversal Lagrangian subbundles $(T\mathcal M,T^\ast\mathcal M)$. On the other hand, since $\d\Omega=0$ and the Poisson structure $\pi$ is trivial, i.e., identically zero, (\ref{eq:act2}) yields  
\begin{equation}
\Omega.(T^\ast\mathcal M,\d,\phi)=(T^\ast\mathcal M,\d,\phi).\label{eq:omeorbit} 
\end{equation}
As opposed to what is stated in part (ii) of Theorem 2.6 in \cite{SX}, this example seems to suggest that the correspondence between quasi-Lie bialgebroids and pairs of a Lagrangian subbundle and of a transversal Dirac structure of a given Courant algebroid, described in the first part of this section is, in general, not one-to-one.
On the other hand, the above example and its generalization, contained in Theorem \ref{thm:Courant}, propound the existence of a one-to-one correspondence between the 
$\Omega^2_c(\mathcal M)$-orbits of quasi-Lie bialgebroids of the type $((T^\ast\mathcal M)_\pi,\d_N,\phi)$ and the $\Omega_c^2(\mathcal M)$-orbits of pairs of a Lagrangian subbundle and of a transversal Dirac structure of the type $({\rm Gr}(\Omega),T^\ast\mathcal M)$ in a Courant algebroid of the type $E=(T\mathcal M)_N\oplus (T^\ast\mathcal M)_\pi$, where $\Omega_c^2(\mathcal M)$ is the (additive) group of the closed 2-forms on $\mathcal M$ acting, on these sets, via (\ref{eq:act2}) and, respectively, (\ref{eq:gauge}).
\end{remark}

We are now ready to present the third proof of Theorem \ref{thm:gim}. We start from a PqN manifold $(\M,\pi,N,\phi)$ to construct the quasi-Lie bialgebroid $((T^*\M)_\pi,\d_N,\phi)$ 
as recalled in Subsection \ref{subsec:qLba}, and therefore the Courant algebroid 
$E= (T\M)_N \oplus (T^*\M)_\pi$ with its Lagrangian subbundle $A^*=T\M$ and its Dirac structure $A=T^*\M$. 
Given a closed 2-form $\Omega$, we apply Formula (\ref{eq:gauge}) to each member of the pair $(T\mathcal M,T^\ast\mathcal M)$, see Remark \ref{rem:actLag}, to get $(L={\rm Gr}(\Omega),T^\ast\mathcal M)$, where $L$ is a \emph{non-integrable} Lagrangian subbundle of $E$. Note that the action of $\Omega$ on the original pair leaves $T^\ast\mathcal M$ fixed. Applying the construction presented in the first part of this section to the new pair $(L,T^\ast\mathcal M)$, we obtain the quasi-Lie bialgebroid $((T^*\M)_\pi,\d_L,\varphi)$ which, thanks to Theorem \ref{thm:Courant}, we are able to identify with $((T^*\M)_\pi,\d_{\widehat N},\widehat\phi)$. Finally, applying Proposition \ref{pro:PXS} to this quasi-Lie bialgebroid we conclude that the quadruplet $(\mathcal M,\pi,\widehat{N},\widehat\phi)$ is a Poisson quasi-Nijenhuis manifold, yielding our (third) proof of Theorem \ref{thm:gim}.  
In this way we are left with proving the following
\begin{theorem}
\label{thm:Courant}
If the 2-form $\Omega$ is closed, 
the quasi-Lie bialgebroid $((T^*\M)_\pi,\d_L,\varphi)$ coincides with
$((T^*\M)_\pi,\d_{\widehat N},\widehat\phi)$, where $\widehat N=N+\pi^\sharp\,\Omega^{\flat}$ and $\widehat\phi$ is given by (\ref{phi-hat}).
\end{theorem}
The first step to prove this theorem is the computation of the bracket (\ref{Courant-bracket}) between two sections of $L$. To this aim, we need two preliminary results.

\begin{lemma}
\label{lem:Koszul}
For any closed 2-form $\Omega$, the following relation holds:
\begin{equation}
\label{eq:Koszul}
[\Omega^\flat X,\Omega^\flat Y]_\pi=\Omega^\flat [X,Y]_{\pi^\sharp\Omega^\flat}+\frac12 i_{X\wedge Y}[\Omega,\Omega]_\pi.
\end{equation}
\end{lemma}

A direct proof of this lemma is given in Appendix B. Another proof can be spelled out along the lines of the following 

\begin{remark}\label{rem:Koszul-bis}
For any 2-form $\Omega$, not necessarily closed, a tedious but straightforward computation which uses \eqref{eq:dorfman} shows that
\begin{equation}
\label{eq:Koszul-bis}
[\Omega^\flat X,\Omega^\flat Y]_\pi=\Omega^\flat [X,Y]^\pi_\Omega+\frac12 i_{X\wedge Y}[\Omega,\Omega]_\pi,
\end{equation}
where
$$
[X,Y]_\Omega^\pi={\mathcal L}^\pi_{\Omega^{\flat} X}{Y} -{\mathcal L}^\pi_{\Omega^{\flat} Y}{X} - \d_\pi (\Omega(X,Y))
$$
and
$$
    \mathcal{L}^{\pi}_{\alpha}Y= \mathcal{L}_{\pi^{\sharp}\alpha}Y+ \pi^{\sharp}(i_Y(\d\alpha)),
$$
for all $X,Y\in\Gamma(T\mathcal M)$ and $\alpha\in\Gamma(T^\ast\mathcal M)$. Then one can obtain (\ref{eq:Koszul}) from the fact that $[X,Y]_\Omega^\pi=[X,Y]_{\pi^\sharp\Omega^\flat}$ 
for all vector field $X,Y$ if $\Omega$ is closed.
\end{remark}

\begin{lemma}
\label{lem:dN-Omega}
For any 2-form $\Omega$, one has that
\begin{equation}
\label{dN-Omega}
\d_N(i_{X\wedge Y}\Omega)=i_X \d_N i_Y\Omega-i_Y \d_N i_X\Omega-i_{X\wedge Y}\d_N\Omega-i_{[X,Y]_N}\Omega.
\end{equation}
\end{lemma}
\begin{proof}
Using (\ref{def:d_N}), we obtain
\begin{equation}
\label{dN-Omega2}
\begin{aligned}
\d_N\Omega(X,Y,Z)&={\mathcal L}_{NX}\langle i_Y\Omega,Z\rangle-{\mathcal L}_{NY}\langle i_X\Omega,Z\rangle
+{\mathcal L}_{NZ}\langle i_X\Omega,Y\rangle\\
&\quad-\Omega([X,Y]_N,Z)+\Omega([X,Z]_N,Y)-\Omega([Y,Z]_N,X)
\\&={\mathcal L}_{NX}\langle i_Y\Omega,Z\rangle-{\mathcal L}_{NY}\langle i_X\Omega,Z\rangle
+\langle\d(\Omega(X,Y)),NZ\rangle\\
&\quad-\langle i_{[X,Y]_N}\Omega,Z\rangle-\langle i_Y\Omega,[X,Z]_N\rangle+\langle i_X\Omega,[Y,Z]_N\rangle
\end{aligned}
\end{equation}
and 
$$
\begin{aligned}
\langle i_X\d_N i_Y\Omega,Z\rangle=\d_N(i_Y\Omega)(X,Z)&={\mathcal L}_{NX}\langle i_Y\Omega,Z\rangle-{\mathcal L}_{NZ}\langle i_Y\Omega,X\rangle-\langle i_Y\Omega,[X,Z]_N\rangle\\
\langle i_Y\d_N i_X\Omega,Z\rangle=\d_N(i_X\Omega)(Y,Z)&={\mathcal L}_{NY}\langle i_X\Omega,Z\rangle-{\mathcal L}_{NZ}\langle i_X\Omega,Y\rangle-\langle i_X\Omega,[Y,Z]_N\rangle.
\end{aligned}
$$
Substituting these relations into (\ref{dN-Omega2}), we find (\ref{dN-Omega}) evaluated on an arbitrary vector field $Z$.
\end{proof}

\begin{lemma}
\label{lem:Courant-bracket}
If $\d\Omega=0$, then  
\begin{equation}
\label{Courant-bracket2}
    \llbracket X +
    \Omega^{\flat} X,Y + 
    \Omega^{\flat}Y\rrbracket= 
[X, Y]_{\widehat N}+
   \Omega^{\flat}[X, Y]_{\widehat N}+ i_{X\wedge Y}\left(\phi+\d_{N} \Omega+\frac{1}{2}[\Omega, \Omega]_\pi\right) 
    .
\end{equation}
\end{lemma}
\begin{proof}
First we compute the vector field component of the left-hand side of (\ref{Courant-bracket2}). By definition (\ref{Courant-bracket}), it is given by
$$
\begin{aligned}
&[X,Y]_N+i_{\Omega^\flat X}(\d_\pi Y)-i_{\Omega^\flat Y}(\d_\pi X)+\d_\pi(\Omega(X,Y))
\\&\qquad =
[X,Y]_N-i_{\Omega^\flat X}(\mathcal{L}_Y\pi)+i_{\Omega^\flat Y}(\mathcal{L}_X\pi)-\pi^\sharp\d(\Omega(X,Y))
\\&\qquad=
[X,Y]_N-\mathcal{L}_Y(i_{\Omega^\flat X}\pi)+i_{\mathcal{L}_Y(\Omega^\flat X)}\pi+\mathcal{L}_X(i_{\Omega^\flat Y}\pi)
-i_{\mathcal{L}_X(\Omega^\flat Y)}\pi-\pi^\sharp\d(\Omega(X,Y))
\\&\qquad=
[X,Y]_N+[\pi^\sharp\Omega^\flat X,Y]+[X,\pi^\sharp\Omega^\flat Y]+\pi^\sharp\left(\mathcal{L}_Y(\Omega^\flat X)-\mathcal{L}_X(\Omega^\flat Y)
-\d(\Omega(X,Y))\right).
\end{aligned}
$$
Now, the last three terms are simply given by 
$$
\begin{aligned}
&\mathcal{L}_Y(i_X\Omega)-\mathcal{L}_X(i_Y\Omega)-(\d\circ i_Y)i_X\Omega=
( i_Y\circ\d)i_X\Omega-\mathcal{L}_X(i_Y\Omega)\\
&\qquad=
i_Y\left((\d\circ i_X)\Omega\right)-i_{[X,Y]}\Omega-i_Y(\mathcal{L}_X\Omega)=
-i_Yi_X\d\Omega-i_{[X,Y]}\Omega=-\Omega^\flat[X,Y]
\end{aligned}
$$
since $\Omega$ is closed. Hence we have shown that the vector field component of the left-hand side of (\ref{Courant-bracket2}) is 
$$
[X,Y]_N+[\pi^\sharp\Omega^\flat X,Y]+[X,\pi^\sharp\Omega^\flat Y]-\pi^\sharp\Omega^\flat[X,Y]=[X,Y]_{N+\pi^\sharp\Omega^\flat}
=[X,Y]_{\widehat N}.
$$
The 1-form component of the left-hand side of (\ref{Courant-bracket2}) is 
\begin{equation}
\label{1-form-comp}
[\Omega^\flat X,\Omega^\flat Y]_\pi+i_{X\wedge Y}\phi-i_{Y}\d_N i_X\Omega+i_{X}\d_N i_Y\Omega-\d_N(\Omega(X,Y)).
\end{equation}
Using Lemma \ref{lem:Koszul} for the first term and Lemma \ref{lem:dN-Omega} for the last three, the sum (\ref{1-form-comp}) turns out to be
$$
\begin{aligned}&
\Omega^\flat [X,Y]_{\pi^\sharp\Omega^\flat}+\frac12 i_{X\wedge Y}[\Omega,\Omega]_\pi+i_{X\wedge Y}\phi+i_{X\wedge Y}\d_N\Omega
+i_{[X,Y]_N}\Omega\\&\quad
= \Omega^{\flat}[X, Y]_{\widehat N}+ i_{X\wedge Y}\left(\phi+\d_{N} \Omega+\frac{1}{2}[\Omega, \Omega]_\pi\right).
\end{aligned}
$$
\end{proof}

We can now 
prove Theorem \ref{thm:Courant}, i.e., 
that the quasi-Lie bialgebroid $((T^*\M)_\pi,\d_L,\varphi)$ coincides 
with
$((T^*\M)_\pi,\d_{\widehat N},\widehat\phi)$, where $\widehat N=N+\pi^\sharp\,\Omega^{\flat}$ and 
$
\widehat\phi=\phi+\d_N\Omega+\frac{1}{2}[\Omega,\Omega]_\pi
$.
First of all, we notice that the identification (\ref{eq:identification}) in this case is simply 
$\tilde X=X+\Omega^\flat X$, where $X\in\Gamma(T\M)$.
By definition (\ref{eq:phi-Courant}) of $\varphi$ and using Lemma \ref{lem:Courant-bracket}, we have that
\begin{align*}
    \varphi(
    X,
    Y,
    Z)&=2 \langle\llbracket \tilde X,\tilde Y\rrbracket, \tilde Z \rangle_E\\
    &=2 \langle\llbracket X + 
    \Omega^{\flat}X,Y + 
    \Omega^{\flat}Y\rrbracket, Z + 
    \Omega^{\flat}Z \rangle_E\\
    &=2 \left\langle [X, Y]_{\widehat N}+ \Omega^{\flat}[X, Y]_{\widehat N}+ i_{X\wedge Y}\left(\phi+\d_{N} \Omega
    +\frac{1}{2}[\Omega, \Omega]_\pi\right),Z + 
    \Omega^{\flat}Z \right\rangle_E\\
    &=\left\langle  \Omega^{\flat}Z,[X, Y]_{\widehat N}\right\rangle+
    \left\langle \Omega^{\flat}[X, Y]_{\widehat N}+ i_{X\wedge Y}\left(\phi+\d_{N} \Omega
    +\frac{1}{2}[\Omega, \Omega]_\pi\right),Z \right\rangle\\
    &=\left\langle i_{X\wedge Y}\left(\phi+\d_{N} \Omega
    +\frac{1}{2}[\Omega, \Omega]_\pi\right),Z \right\rangle\\
    &=\widehat\phi(X,Y,Z)
    \end{align*}
for all $X,Y,Z\in\Gamma(T\M)$. So we are left with showing that $\d_L$ acts as $\d_{\widehat{N}}$. But this immediately follows from
$$
\rho(X + \Omega^{\flat}X)=NX+\pi^\sharp \Omega^{\flat}X=\widehat NX
\qquad\mbox{and}\qquad
[X + 
    \Omega^{\flat}X,Y + 
    \Omega^{\flat}Y]_L=[X,Y]_{\widehat N}+\Omega^{\flat}[X,Y]_{\widehat N}.
$$
Indeed, if $\psi\in\Gamma(\bigwedge^k A)$ corresponds to $\tilde\psi\in\Gamma(\bigwedge^k L^*)$,
then from (\ref{eq:dL}) we obtain
$$
\begin{aligned}
	(\d_L 
	\psi)(X_1,\ldots,X_{k+1})&= \sum_{i=1}^{k+1}(-1)^{i+1}\rho({\tilde X}_i) \left(\tilde
	\psi({\tilde X}_1,\ldots,\hat{{\tilde X}}_i,\ldots,{\tilde X}_{k+1}) \right) \\
	&\ \ +\sum_{i <j} (-1)^{i+j}
	\tilde\psi([{\tilde X}_i,{\tilde X}_j]_L, {\tilde X}_1 \ldots, \hat{{\tilde X}}_i, \ldots ,\hat{{\tilde X}}_j, \ldots , {\tilde X}_{k+1})\\
	&= \sum_{i=1}^{k+1}(-1)^{i+1}\left(\widehat N X_i\right) \left(\psi({X}_1,\ldots,\hat{{X}}_i,\ldots,{X}_{k+1}) \right) \\
	&\ \ +\sum_{i <j} (-1)^{i+j}
	\psi([{X}_i,{X}_j]_{\widehat N}, {X}_1 \ldots, \hat{{X}}_i, \ldots ,\hat{{X}}_j, \ldots , {X}_{k+1})\\
	&=(\d_{\widehat N} 
	\psi)(X_1,\ldots,X_{k+1}).
\end{aligned}
$$

\section*{Appendix A: Definition of Courant algebroid}
\renewcommand{\theequation}{A\arabic{equation}}
 \setcounter{equation}{0}

In this appendix we will recall the definition of Courant algebroid following \cite{LWX} --- see also \cite{Courant1990}, where this definition appeared for the first time. To this end, let $\mathcal M$ be a manifold. 
A Courant algebroid (over $\mathcal M$) is a quadruplet $(E,\langle\cdot,\cdot\rangle_E,\llbracket\cdot,\cdot\rrbracket,\rho)$, where
\begin{enumerate}
\item[(1)] $E$ is a vector bundle over $\mathcal M$;
\item[(2)] $\langle\cdot,\cdot\rangle_E:\Gamma(E)\times\Gamma(E)\rightarrow\mathbb R$ is a symmetric, non-degenerate, and $C^\infty(\mathcal M)$-bilinear form;
\item[(3)] $\llbracket\cdot,\cdot\rrbracket:\Gamma(E)\times\Gamma(E)\rightarrow\Gamma(E)$ is an $\mathbb R$-bilinear, skew-symmetric bracket and
\item[(4)] $\rho:E\rightarrow TM$ is a bundle-map, called the \emph{anchor} of the Courant algebroid, inducing the $\mathbb R$-linear operator $\mathcal D:C^\infty(M)\rightarrow\Gamma(E)$, via the 
formula
\begin{equation}
\langle\mathcal Df,A\rangle_E:=\frac{1}{2}\rho(A)(f),\qquad\forall A\in\Gamma(E),\;f\in C^\infty(\mathcal M),\label{eq:calD}
\end{equation}
\end{enumerate}
satisfying the following compatibility conditions. For all $A,B,C\in\Gamma(E)$ and $f,g\in C^\infty(\mathcal M)$,
\begin{enumerate}
\item[(i)] $\rho$ is \emph{bracket-compatible}, i.e., $\rho(\llbracket A,B\rrbracket)=[\rho(A),\rho(B)]$;
\item[(ii)]
$
\llbracket\llbracket A,B\rrbracket,C\rrbracket+\llbracket\llbracket B,C\rrbracket,A\rrbracket+\llbracket\llbracket C,A\rrbracket,B\rrbracket
=\frac{1}{3}\mathcal D\big(\langle\llbracket A,B\rrbracket,C\rangle_E+\langle\llbracket B,C\rrbracket,A\rangle_E
+\langle\llbracket C,A\rrbracket,B\rangle_E\big);
$
\item[(iii)] $\llbracket A,fB\rrbracket=f\llbracket A,B\rrbracket+\rho(A)(f)B-\langle A,B\rangle_E\,\mathcal D(f)$;
\item[(iv)] $\langle\mathcal D(f),\mathcal D(g)\rangle_E=0$;
\item[(v)] $\rho(A)\langle B,C\rangle_E=\langle\llbracket A,B\rrbracket+\mathcal D\langle A,B\rangle_E,C\rangle_E+\langle B,\llbracket A,C\rrbracket+\mathcal D\langle A,C\rangle_E\rangle_E$.
\end{enumerate}

A few comments are in order.
\begin{remark} Note that:
\begin{itemize}
\item
as shown in \cite{Uchino2002}, conditions (iii) and (iv) follow from the other conditions;
\item $\mathcal D$ defined in \eqref{eq:calD} is a differential operator in the sense that it satisfies the Leibniz identity;
\item $\llbracket\cdot,\cdot\rrbracket$ is not a Lie bracket since it does not satisfies the Jacobi's identity, see item $(ii)$ of the previous list.
\end{itemize}
\end{remark}

\begin{example}[Quadratic Lie algebras]
A Courant algebroid over a point, i.e., if $\mathcal M=\{pt\}$, is the same as a quadratic Lie algebra, i.e., a Lie algebra endowed with a symmetric, non-degerate and $ad$-invariant bilinear form. In fact, if $\mathcal M=\{pt\}$, then $\rho=0$, which forces the condition $\mathcal D=0$. In this way $\llbracket\cdot,\cdot\rrbracket$ becomes a Lie bracket and $\langle\cdot,\cdot\rangle_E$ becomes a non-degenerate, symmetric and $ad$-invariant bilinear form, see items (ii) and 
(v) above.
\end{example}
\begin{example}[Standard Courant algebroid]\label{ex:standardCou} In this case $E=\mathbb T\mathcal M=T\mathcal M\oplus T^\ast\mathcal M$, $\rho:E\rightarrow T\mathcal M$ is the projection on the first summand, and $\langle X+\alpha,Y+\beta\rangle_E=\frac{1}{2}(\langle\alpha,Y\rangle+\langle\beta,X\rangle)$. In particular, $\langle\mathcal D f,X+\alpha\rangle_E=\frac{1}{2}X(f)$. Moreover,
\begin{eqnarray}\label{eq:bbcourant}
\llbracket X+\alpha,Y+\beta\rrbracket&=&[X,Y]+\mathcal L_X\beta-\mathcal L_Y\alpha+\frac{1}{2}\d(i_Y\alpha-i_X\beta)\\
&=&[X,Y]+i_X\d\beta-i_Y\d\alpha+\frac{1}{2}\d(i_X\beta-i_Y\alpha)\nonumber
\end{eqnarray}
for all $f\in C^\infty(\mathcal M)$ and $X+\alpha,Y+\beta\in\Gamma(E)$. Note that the pair $(T\mathcal M,T^\ast \mathcal M)$ forms a Lie bialgebroid, i.e., a Lie quasi-bialgebroid such that $\phi=0$, see Definition \ref{def:qLba}.
It turns out that $T\mathcal M$ carries the structure of Lie algebroid defined by the standard Lie brackets on vector fields and by the identity as anchor map, while $T^\ast\mathcal M$ carries the trivial Lie algebroid structure, i.e., with zero Lie bracket and zero anchor map.

The bracket (\ref{eq:bbcourant}) can be modified by twisting it with the term $i_{X\wedge Y}\phi$, where $\phi$ is any closed 3-form on $\M$. The resulting structure is called {\em twisted} Courant algebroid and it is denoted with $\mathbb T\M^\phi$. Twisted Courant algebroids were introduced by \u{S}evera in \cite{Severa}, who proved that 
a Courant algebroid $E$ fits into the exact sequence 
\[
0\longrightarrow T^\ast\M\stackrel{\rho^\ast}{\longrightarrow} E\stackrel{\rho}\longrightarrow T\M\longrightarrow 0,
\]
if and only if $E$ is isomorphic to $\mathbb T\M^\phi$ for some closed 3-form. In the exact sequence $\rho$ denotes the anchor of $E$.
\end{example}

\section*{Appendix B: Proof of Lemma \ref{lem:Koszul}} 
\renewcommand{\theequation}{B\arabic{equation}}
 \setcounter{equation}{0}
We will now prove that, for all vector fields $X,Y$, the 
identity 
\[
[\Omega^\flat X,\Omega^\flat Y]_\pi=\Omega^\flat [X,Y]_{\pi^\sharp\Omega^\flat}+\frac12 i_{X\wedge Y}[\Omega,\Omega]_\pi
\]
holds true.
Since $\d\Omega=0$, we can use (\ref{eq:monella}) to compute $[\Omega,\Omega]_\pi=\d_{\pi^\sharp\Omega^\flat}\Omega$. Then 
(\ref{def:d_N}) entails that
\begin{equation}
\label{koszul-dorfman}
[\Omega,\Omega]_\pi(X,Y,Z)=\sum_{\circlearrowleft(X,Y,Z)}\left({\mathcal L}_{\pi^\sharp\Omega^\flat Z}\left(\Omega(X,Y)\right)
-\Omega\left([X,Y]_{\pi^\sharp\Omega^\flat},Z\right)\right)
\end{equation}
for any vector fields $X$, $Y$, and $Z$. To show that (\ref{eq:Koszul}) holds, we define
$$
A(X,Y,Z)=\langle[\Omega^\flat X,\Omega^\flat Y]_\pi-\Omega^\flat [X,Y]_{\pi^\sharp\Omega^\flat},Z\rangle
$$
and we first show that $A(X,Y,Y)=0$, so that $A(X,Y,Z)=-A(X,Z,Y)$ for all $X,Y,Z\in\Gamma(T\M)$. Indeed,
\begin{equation*}
\begin{aligned}
A(X,Y,Y)
     &=\left\langle {\mathcal L}_{\pi^\sharp\Omega^\flat X}(\Omega^\flat Y)- {\mathcal L}_{\pi^\sharp\Omega^\flat Y}(\Omega^\flat X)
-\d\langle\Omega^\flat Y,\pi^\sharp\Omega^\flat X\rangle\right.
\\&\quad\left.-\Omega^\flat\left([\pi^\sharp\Omega^\flat X,Y]+[X,\pi^\sharp\Omega^\flat Y]
-\pi^\sharp\Omega^\flat [X,Y]\right),Y\right\rangle\\
     &={\mathcal L}_{\pi^\sharp\Omega^\flat X}\langle \Omega^\flat Y,Y\rangle
     -\cancel{\langle\Omega^\flat Y,[\pi^\sharp\Omega^\flat X,Y]\rangle}
-{\mathcal L}_{\pi^\sharp\Omega^\flat Y}\langle \Omega^\flat X,Y\rangle
+\langle\Omega^\flat X,[\pi^\sharp\Omega^\flat Y,Y]\rangle
\\
&\quad -{\mathcal L}_Y\langle\Omega^\flat Y,\pi^\sharp\Omega^\flat X\rangle
+\cancel{\langle\Omega^\flat Y,[\pi^\sharp\Omega^\flat X,Y]\rangle}
+\langle\Omega^\flat Y,[X,\pi^\sharp\Omega^\flat Y]\rangle
-\langle\Omega^\flat \pi^\sharp\Omega^\flat Y,[X,Y]\rangle\\
    &= -{\mathcal L}_{\pi^\sharp\Omega^\flat Y}\langle \Omega^\flat X,Y\rangle-\langle\Omega^\flat [\pi^\sharp\Omega^\flat Y,Y],X\rangle
-{\mathcal L}_Y\langle\Omega^\flat Y,\pi^\sharp\Omega^\flat X\rangle
-\langle\Omega^\flat Y,{\mathcal L}_{\pi^\sharp\Omega^\flat Y}X\rangle\\
&\quad 
+\langle\Omega^\flat \pi^\sharp\Omega^\flat Y,{\mathcal L}_YX\rangle\\
 &=-\cancel{{\mathcal L}_{\pi^\sharp\Omega^\flat Y}\langle \Omega^\flat X,Y\rangle}
-\langle\Omega^\flat [\pi^\sharp\Omega^\flat Y,Y],X\rangle
-\bcancel{{\mathcal L}_Y\langle\Omega^\flat Y,\pi^\sharp\Omega^\flat X\rangle}
-\cancel{{\mathcal L}_{\pi^\sharp\Omega^\flat Y}\langle\Omega^\flat Y,X\rangle}
\\
&\quad 
+\langle{\mathcal L}_{\pi^\sharp\Omega^\flat Y}(\Omega^\flat Y),X\rangle
+\bcancel{{\mathcal L}_Y\langle\Omega^\flat \pi^\sharp\Omega^\flat Y,X\rangle}
-\langle{\mathcal L}_Y(\Omega^\flat \pi^\sharp\Omega^\flat Y),X\rangle
\\
&=\langle -\Omega^\flat [\pi^\sharp\Omega^\flat Y,Y]+{\mathcal L}_{\pi^\sharp\Omega^\flat Y}(\Omega^\flat Y)
-{\mathcal L}_Y(\Omega^\flat \pi^\sharp\Omega^\flat Y) ,X\rangle,
\end{aligned}
\end{equation*}
and we have that
\begin{equation*}
\begin{aligned}
 -\Omega^\flat [\pi^\sharp\Omega^\flat Y,Y]+{\mathcal L}_{\pi^\sharp\Omega^\flat Y}(\Omega^\flat Y)
-{\mathcal L}_Y(\Omega^\flat \pi^\sharp\Omega^\flat Y) 
&=-i_{{\mathcal L}_{\pi^\sharp\Omega^\flat Y}Y}(\Omega)+{\mathcal L}_{\pi^\sharp\Omega^\flat Y}(i_Y\Omega)
-{\mathcal L}_Y(i_{\pi^\sharp\Omega^\flat Y}\Omega)\\
&=i_Y\left({\mathcal L}_{\pi^\sharp\Omega^\flat Y}\Omega\right)-(\d\circ i_Y+i_Y\circ\d)\left(i_{\pi^\sharp\Omega^\flat Y}\Omega\right)\\
&=i_Y\left({\mathcal L}_{\pi^\sharp\Omega^\flat Y}\Omega\right)
-\left(i_Y\circ\d\circ i_{\pi^\sharp\Omega^\flat Y}\right)\Omega\\
&=i_Y\left({\mathcal L}_{\pi^\sharp\Omega^\flat Y}\Omega\right)
-i_Y\left({\mathcal L}_{\pi^\sharp\Omega^\flat Y}\Omega- i_{\pi^\sharp\Omega^\flat Y}\d\Omega\right),
\end{aligned}
\end{equation*}
which vanishes if $\Omega$ is closed. Now,
\begin{equation*}
\begin{aligned}
2A(X,Y,Z)&=
A(X,Y,Z)-A(X,Z,Y)
\\
     &=
     \left\langle {\mathcal L}_{\pi^\sharp\Omega^\flat X}(\Omega^\flat Y)- {\mathcal L}_{\pi^\sharp\Omega^\flat Y}(\Omega^\flat X)
-\d\langle\Omega^\flat Y,\pi^\sharp\Omega^\flat X\rangle 
-\Omega^\flat[X,Y]_{\pi^\sharp\Omega^\flat},Z\right\rangle\\
&\quad-
\left\langle {\mathcal L}_{\pi^\sharp\Omega^\flat X}(\Omega^\flat Z)- {\mathcal L}_{\pi^\sharp\Omega^\flat Z}(\Omega^\flat X)
-\d\langle\Omega^\flat Z,\pi^\sharp\Omega^\flat X\rangle
-\Omega^\flat[X,Z]_{\pi^\sharp\Omega^\flat},Y\right\rangle\\
     &=     {\mathcal L}_{\pi^\sharp\Omega^\flat X}\langle \Omega^\flat Y,Z\rangle
     -\langle\Omega^\flat Y,[\pi^\sharp\Omega^\flat X,Z]\rangle
    -{\mathcal L}_{\pi^\sharp\Omega^\flat Y}\langle \Omega^\flat X,Z\rangle
     +\langle\Omega^\flat X,[\pi^\sharp\Omega^\flat Y,Z]\rangle
 \\&\quad -\left\langle\d\langle\Omega^\flat Y,\pi^\sharp\Omega^\flat X\rangle,Z\right\rangle
     -\Omega\left([X,Y]_{\pi^\sharp\Omega^\flat},Z\right)\\
&\quad -\langle {\mathcal L}_{\pi^\sharp\Omega^\flat X}(\Omega^\flat Z),Y\rangle
    +{\mathcal L}_{\pi^\sharp\Omega^\flat Z}\langle \Omega^\flat X,Y\rangle
     -\langle\Omega^\flat X,[\pi^\sharp\Omega^\flat Z,Y]\rangle
 \\&\quad +\left\langle\d\langle\Omega^\flat Z,\pi^\sharp\Omega^\flat X\rangle,Y\right\rangle
     +\Omega\left([X,Z]_{\pi^\sharp\Omega^\flat},Y\right)\\
&=     {\mathcal L}_{\pi^\sharp\Omega^\flat X}\left(\Omega(Y,Z)\right)
+  {\mathcal L}_{\pi^\sharp\Omega^\flat Y}\left(\Omega(Z,X)\right)
+  {\mathcal L}_{\pi^\sharp\Omega^\flat Z}\left(\Omega(X,Y)\right)
\\&\quad -\Omega\left([X,Y]_{\pi^\sharp\Omega^\flat},Z\right)
-\Omega\left([Z,X]_{\pi^\sharp\Omega^\flat},Y\right)
\\&\quad  -\langle\Omega^\flat Y,[\pi^\sharp\Omega^\flat X,Z]\rangle
     +\langle\Omega^\flat X,[\pi^\sharp\Omega^\flat Y,Z]\rangle
-\left\langle\d\langle\Omega^\flat Y,\pi^\sharp\Omega^\flat X\rangle,Z\right\rangle
\\&\quad 
-\langle {\mathcal L}_{\pi^\sharp\Omega^\flat X}(\Omega^\flat Z),Y\rangle
     -\langle\Omega^\flat X,[\pi^\sharp\Omega^\flat Z,Y]\rangle
+\left\langle\d\langle\Omega^\flat Z,\pi^\sharp\Omega^\flat X\rangle,Y\right\rangle.
          \end{aligned}
\end{equation*}
The last six terms are equal to
\begin{equation*}
\begin{aligned}
&\langle\Omega^\flat Y,{\mathcal L}_Z(\pi^\sharp\Omega^\flat X)\rangle
     -\Omega\left([\pi^\sharp\Omega^\flat Y,Z]+[Y,\pi^\sharp\Omega^\flat Z],X\right)
-{\mathcal L}_Z\langle\Omega^\flat Y,\pi^\sharp\Omega^\flat X\rangle
\\&\quad 
-\cancel{\langle (\d\circ i_{\pi^\sharp\Omega^\flat X})(i_Z\Omega),Y\rangle}
-\langle (i_{\pi^\sharp\Omega^\flat X}\circ\d)(i_Z\Omega),Y\rangle
+\cancel{\left\langle\d\langle i_Z\Omega,\pi^\sharp\Omega^\flat X\rangle,Y\right\rangle}
\\&=-\langle{\mathcal L}_Z(\Omega^\flat Y),\pi^\sharp\Omega^\flat X\rangle
     -\Omega\left([\pi^\sharp\Omega^\flat Y,Z]+[Y,\pi^\sharp\Omega^\flat Z],X\right)
+\langle (i_{Y}\circ\d)(i_Z\Omega),\pi^\sharp\Omega^\flat X\rangle
\\&=-\langle{\mathcal L}_Z(i_Y\Omega),\pi^\sharp\Omega^\flat X\rangle
     -\Omega\left([\pi^\sharp\Omega^\flat Y,Z]+[Y,\pi^\sharp\Omega^\flat Z],X\right)
+\langle i_{Y}({\mathcal L}_Z\Omega-i_Z(\d\Omega)),\pi^\sharp\Omega^\flat X\rangle
\\&=-\langle i_{[Z,Y]}\Omega,\pi^\sharp\Omega^\flat X\rangle
     -\Omega\left([\pi^\sharp\Omega^\flat Y,Z]+[Y,\pi^\sharp\Omega^\flat Z],X\right)
\\&= - \Omega\left([Y,Z]_{\pi^\sharp\Omega^\flat},X\right),
        \end{aligned}
\end{equation*}
where we have used again the fact that $\d\Omega=0$. So we conclude that 
$$
2A(X,Y,Z)= \sum_{\circlearrowleft(X,Y,Z)}\left({\mathcal L}_{\pi^\sharp\Omega^\flat X}\left(\Omega(Y,Z)\right)-\Omega\left([X,Y]_{\pi^\sharp\Omega^\flat},Z\right)\right)
.
$$
The comparison between this formula and (\ref{koszul-dorfman}) ends the proof.

%

\end{document}